\documentclass[12pt,reqno]{amsart}
\usepackage{amsmath}
\usepackage{amssymb}
\usepackage{amsthm}
\usepackage{tikz-cd}
\usepackage{enumitem}
\usepackage[hidelinks]{hyperref}

\newtheorem*{theorem*}{Theorem}
\newtheorem{theorem}{Theorem}[section]
\newtheorem{proposition}[theorem]{Proposition}
\newtheorem{lemma}[theorem]{Lemma}
\newtheorem{conjecture}[theorem]{Conjecture}
\newtheorem{corollary}[theorem]{Corollary}
\theoremstyle{definition}
\newtheorem{notation}[theorem]{Notation}
\newtheorem{definition}[theorem]{Definition}
\newtheorem{remark}[theorem]{Remark}
\newtheorem{example}[theorem]{Example}
\newtheorem{question}[theorem]{Question}

\newcommand{\mc}{\mathcal}
\newcommand{\mf}{\mathbf}
\newcommand{\mb}{\mathbb}
\newcommand{\vep}{\varepsilon}

\newcommand{\on}{\operatorname}

\usepackage[margin=0.75in]{geometry}
\begin{document}
\title{The Seshadri Constants of Tangent Sheaves on Toric Varieties}
\author{Chih-Wei Chang}
\address{Department of Mathematics, National Taiwan University, Taipei, Taiwan}
\email{cwchang0219@ntu.edu.tw}
\begin{abstract}
In this paper, we investigate the Seshadri constant $\vep(X,T_X;p)$ of the tangent sheaf $T_X$ of a proper $\mathbb Q$-factorial toric variety $X$. We show that $\vep(X,T_X;1)>0$ if and only if the following statement holds true: 
if $a_1v_1+\cdots +a_kv_k=0$ where each $a_i$ is a positive real number and each $v_i$ is the primitive generator of some ray in the fan $\Delta$ that defines $X$, then $k\geq \dim X+1$. Based on the result, we show that a smooth projective toric variety $X$ with $\vep(X,T_X;p)>0$ for some $p\in X$ is isomorphic to the projective space, confirming a special case of the conjecture proposed by M. Fulger and T. Murayama.
\end{abstract}
\maketitle
Throughout the paper, all varieties are irreducible and defined over an algebraically closed field $\mf k$ of any characteristic. Unless stated otherwise, points refer to closed points.
The classical Seshadri constant of a nef invertible sheaf $\mc L$ on a projective variety $X$ at a closed point $p$, introduced by Demailly \cite{MR1178721},  is defined to be
\begin{align*}
\vep(\mc L,p)&=\sup\{t\in\mb R_{\geq 0}\mid\pi^*c_1(\mc L)-tE\ \text{is nef},\ \pi\colon\on{Bl}_p(X)\rightarrow X,\ E=\on{exc}(\pi)\}\\
             &=\inf
\{\frac{c_1(\mc L)\cdot C}{\on{mult}_p(C)}\mid C\ \text{is an irreducible curve passing through}\ p\}\in \mb R_{\geq 0}.
\end{align*}
The Seshadri constants of nef $\mb Q$-Cartier $\mb Q$-divisors are defined in the same way.
These constants measure the local positivity of $\mc L$ and have a close relation with the separation of jets of $\mc L$ and of the adjoint bundle $\omega_X\otimes_{\mc O_X}\mc L$. More precisely, the
Seshadri criterion (Theorem 1.4.12, Lazarsfeld 2004\cite{MR2095471}) asserts that a Cartier divisor $D$ on a projective variety $X$ is ample if and only if $\inf_{p\in X}\{\vep(D,p)\}>0$, and,  if $X$ is smooth projective, $\omega_X\otimes_{\mc O_X}\mc L$ separates $\ell$-jets at $p$ if $\vep(\mc L,p)>\dim X+\ell$ \cite{MR1178721}.  See \cite[Chapter 5]{MR2095471} or \cite{https://doi.org/10.48550/arxiv.0810.0728} for more detailed introductions. 
Despite the fact that the classical Seshadri constants are fundamental and richly structured, there are few studies on the possible generalizations. The most notable works in this direction, among others, are Beltrametti–Schneider–Sommese \cite{MR1248115,MR1360498} and Hacon \cite{MR1779893} on the Seshadri constants of ample vector bundles.
Several years ago, Fulger and Murayama defined the Seshadri constants of not necessarily nef coherent sheaves on projective varieties in \cite{FM21}.
To state the definition, we recall that
for any coherent sheaf $\mc G$ on a smooth projective curve $C$, the minimal slope of $\mc G$ is defined to be
\[
\mu_{\min}(\mc G):=\min\{\mu(\mc H)\mid \mc H\ \text{is a quotient of}\ \mc G,\,\mu(\mc H):=\frac{\deg{\mc H}}{\on{rk}\,\mc H},\ \deg(\mc H):=\chi(\mc H)-\on{rk}\,\mc H\cdot\chi(\mc O_C)\}
\]
with the convention that $\mu(\mc H)=\infty$ when $\mc H$ is torsion.
Following \cite{FM21}, we define
\[
\overline{\mu}_{\min}(\mc G):=\left\{\begin{array}{ll}
\mu_{\min}(\mc G) &,\, \text{if}\ \on{char}(\mf k)=0\\
\lim\limits_{n\rightarrow\infty}\frac{\mu_{\min}((F^n)^*\mc G)}{p^n}&,\, \text{if}\ \on{char}(\mf k)=p
\end{array}\right.,
\]
where $F\colon C\rightarrow C$ is the absolute Frobenius morphism.
The Seshadri constant of a coherent sheaf $\mc F$ on a projective variety $X$ at a closed point $p\in X$ is defined to be (\cite[Section 3]{FM21})
  \begin{multline*}\label{def}\tag{1.1}
  \vep(X,\mc F;p):=\inf_{p\in C\subseteq X}\{\frac{\overline{\mu}_{\min}(\nu^*\mc F)}{\mathrm{mult}_p(C)}\mid C\ \text{is an irreducible curve,}\\ \nu\colon\widetilde C\rightarrow C\ \text{is the normalization}\}\in\mathbb R\cup\{\pm\infty\}.
  \end{multline*} 
We usually write $\vep(\mc F,p)$ if there is no confusion. These constants agree with the classical Seshadri constants when $\mc F$ is a nef invertible sheaf, and it has been shown in \cite{FM21} that they behave similarly to their classical counterparts:  
\begin{itemize}
\item $\inf\{\vep(\mc F,p)\mid p\in X\}>0$ if and only if $\mc F$ is ample (\cite[Theorem 3.11]{FM21}).
\item Suppose $X$ is smooth of dimension $n$ and $\mc F$ is an ample locally free sheaf of rank $r$. If $\vep(\mc F,p)>\frac{n+s}{\ell+r}$, then $\omega_X\otimes S^\ell\mc F\otimes\det\mc F$ 
separates $s$-jets at $p$ (\cite[Proposition 5.7]{FM21}). 
\end{itemize}

Besides being useful in studying the local positivity of coherent sheaves, Seshadri constants also play a role in characterizing the projective spaces.
Bauer and Szemberg showed in \cite[Theorem 1.7]{MR2533311} that a complex Fano manifold is isomorphic to the projective space if the Seshadri constant of $-K_X$ is larger than $\dim X$ at some point. Later, Liu and Zhuang generalized the theorem to complex $\mb Q$-Fano varieties \cite[Theorem 2]{MR3803780}. Fulger and Murayama asked a question in higher rank similar to these observations:
\begin{conjecture}[{\cite[Conjecture 4.9]{FM21}}]\label{conj}
Let $X$ be a smooth projective variety defined over an algebraically closed field. If there exists $p\in X$
such that $\vep(T_X, p)>0$, then $X\simeq \mb P^n$.
\end{conjecture}
In other words, it is expected that the Seshadri constants of the tangent bundles cannot strictly lie between $0$ and $1$.  
The conjecture has been verified when $X$ is a smooth projective surface or a Fano manifold \cite[Proposition 4.8 and Corollary 4.12]{FM21}.
The purpose of this paper is to study Conjecture~\ref{conj} when $X$ is toric. Our main results are as follows:

\begin{theorem}\label{thm1}
Suppose that $X$ is a smooth projective toric variety. If $\vep(T_X,p)>0$ for some $p\in X$, then $X\simeq \mb P^n$.
\end{theorem}

\begin{corollary}[= Proposition~\ref{prop3}]\label{cor0}
Suppose $X$ is a smooth projective toric variety.  Then 
   \[
   \vep(T_X,p)=\min_{\rho\in\Delta(1)}\{\vep(D_\rho,p)\}.
   \]
\end{corollary}
The theorem is false for singular varieties, even for projective threefolds with terminal singularities:
\begin{example}[=Example~\ref{Example:terminal_threefold}]
   There exists a toric projective threefold $X$ with terminal singularities such that $\varepsilon(T_X,1)>0$ but $X\not\simeq \mb P^n$.
\end{example} 
This suggests that, in the singular setting, one might need to replace $T_X$ by something else (for example, $T_X(-D)$ for some suitably defined effective torus invariant $\mb Q$-divisor $D$).  

If $X$ is a projective $\mb Q$-factorial toric variety which is possibly singular, we still have $\vep(T_X,p)\geq\min_{\rho\in\Delta(1)}\{\vep(D_\rho,p)\}$ (see Lemma~\ref{lemma1} and the proof of Proposition~\ref{prop2}). The author is unaware of any projective $\mb Q$-factorial toric variety $X$ that violates Corollary~\ref{cor0}. It is worth mentioning that an explicit formula for the Seshadri constants of nef toric vector bundles at the torus invariant points is given in \cite[Proposition 3.2]{HMP}. On the other hand, when $p$ is not a torus invariant point, it is not easy to calculate the Seshadri constants even for torus invariant divisors.
For example, if $H$ is the ample generator of the Picard group of $X=\mb P(a,b,c)$ and $1\in X$ is the identity of its dense torus, then $\vep(H,1)=1/\sqrt{abc}$ is equivalent to Nagata's conjecture for $abc$ points (\cite[Proposition 5.2]{MR2784746}).

The proof of Theorem~\ref{thm1} is divided into two parts: $p\in T$ or $p\in X\setminus T$. If $p\in X\setminus T$, we first show that the dual of the conormal short exact sequence of a $T$-invariant prime divisor $Y$ inside the smooth projective toric variety $X$ splits, and then apply an induction argument. To deal with the case $p\in T$ (may assume $p=1$), we establish a combinatorial criterion for when $\vep(X,T_X;1)>0$ (Theorem~\ref{combinatoric}).
Combining with \cite[Proposition 3.2]{MR1133869} on the existence of certain primitive collections, we show that the only smooth projective toric variety with $\vep(X,T_X;1)>0$ is the projective space. For this part, if $\on{char}\,\mf k=0$, we note that one can simply apply \cite[Corollary 0.4(11)]{MR1929792} to get the same result (see \cite[Proposition 4.8(2)]{FM21} for details).
The advantage of our argument is that it does not depend on the characteristic of $\bf k$. 

\section{A combinatorial criterion}
The main result of this section is Theorem~\ref{combinatoric}. 
Instead of repeating \cite[Section 3]{FM21}, we will only recall some of the lemmas that are necessary for our calculations.
\begin{lemma}\label{lemma1}
Let $X$ be a projective variety.
\begin{enumerate}[label={\upshape(\arabic*)}]
\item \label{lemma1-1} ({\cite[Lemma 3.31]{FM21}})
If
  \[
  \mathcal G\rightarrow\mathcal F\rightarrow\mathcal H\rightarrow 0
  \]
  is an exact sequence of coherent sheaves on $X$, then for all $p\in X$,
  \[
  \vep(\mathcal F,p)\geq\min\{\vep(\mathcal G,p),\vep(\mathcal H,p)\}.
  \]
  Note that it is true even when $\infty$ occurs.
  If $\mathcal F=\mathcal G\oplus\mathcal H$, then "$=$" holds.

\item \label{lemma1-2}
If $\mathcal V\rightarrow\mathcal Q$ is a surjective morphism of coherent sheaves on $X$, then $\varepsilon(\mathcal Q,p)\geq \varepsilon(\mathcal V,p)$ for all $p\in X$. More generally, if $\mc V|_U\rightarrow\mc Q|_U$ is surjective for some Zariski open subset $U$ of $X$ that contains $p$, then $\varepsilon(\mathcal Q,p)\geq \varepsilon(\mathcal V,p)$.

\item \label{lemma1-3} Suppose $f\colon Y\rightarrow X$ is a birational morphism such that $f|_{f^{-1}(U)}\colon f^{-1}(U)\rightarrow U$ is an isomorphism for some Zariski open subset $U$ containing $p$. Then $\vep(X,\mc F;p)=\vep(Y,f^*\mc F;p')$ for any coherent sheaf $\mc F$ on $X$. Here $p'$ is the preimage of $p$ under $f|_{f^{-1}(U)}$.
\end{enumerate}
\end{lemma}
\begin{proof}

For (2), consider the exact sequence
\[
    \mc V\overset{\alpha}{\rightarrow} \mc Q\rightarrow \on{coker}(\alpha)\rightarrow 0.
\]
Note that the support of $\on{coker}(\alpha)$ is contained in $X\setminus U$ and apply $(1)$.

For (3), let $C$ be an irreducible curve on $X$ passing through $p$, and let $C'$ be its strict transform on $Y$. Note that if $\nu\colon\widetilde{C}\rightarrow C'$ is the normalization, then so is the composition $\widetilde{C}\overset{\nu}{\rightarrow}C'\overset{f}{\rightarrow} C$ and
\[
\frac{\overline{\mu}_{\min}((f\circ\nu)^*\mc F)}{\on{mult}_p(C)}=\frac{\overline{\mu}_{\min}(\nu^*(f^*\mc F))}{\on{mult}_{p'}(C')}.
\]
Hence $\vep(X,\mc F;p)\geq\vep(Y,f^*\mc F;p')$. Since every irreducible curve $C'$ in $Y$ passing through $p'$ arises in this way, we actually have $\vep(X,\mc F;p)=\vep(Y,f^*\mc F;p')$.
\end{proof}

\begin{remark}
Lemma~\ref{lemma1} only depends on the validity of the formula 
\[
\vep(X,\mc F;p)=\inf_{p\in C\subseteq X}\{\frac{\overline{\mu}_{\min}(\nu^*\mc F)}{\mathrm{mult}_p(C)}\mid C\ \text{is an irreducible curve,}\\ \nu\colon\widetilde C\rightarrow C\ \text{is the normalization}\}.
\]
Hence, it is still true for proper varieties if we adopt the following definition of Seshadri constants.
\end{remark}

\begin{definition}[Seshadri constants for proper varieties]\label{Def:Seshadri_Constant_Proper}
  The Seshadri constant of a coherent sheaf $\mc F$ on a proper variety $X$ at a closed point $p\in X$ is defined to be
  \begin{align*}
  \vep(X,\mc F;p)&:=\inf_{p\in C\subseteq X}\{\frac{\overline{\mu}_{\min}(\nu^*\mc F)}{\mathrm{mult}_p(C)}\mid C\ \text{is an irreducible curve, }\nu\colon\widetilde C\rightarrow C\ \text{is the normalization}\}\\
    &\in\mathbb R\cup\{\pm\infty\}.
  \end{align*}
\end{definition}

We refer to \cite{MR1234037} and \cite{MR2810322} for basic knowledge of toric geometry. 
\begin{notation}\label{notaion}
\begin{enumerate}[label={\upshape(\arabic*)}]
\item
$N\simeq\mb Z^n$ will be a lattice and $M:=\on{Hom}(N,\mb Z)$ denotes the dual lattice.
The convex cone in $N_\mb R:= N\otimes\mb R$ generated by $\{v_1,\cdots,v_k\}\subseteq N$ is denoted by $\langle v_1,\cdots,v_k\rangle$. 
A strongly convex rational polyhedral cone $\sigma\subseteq N_\mb R$ is a cone of the form
$\langle v_1,\cdots,v_k\rangle$ for some $\{v_1,\cdots,v_k\}\subseteq N$ such that $\sigma$ does not contain any non-trivial linear subspace of $N_\mb R$.
  \item  A fan $\Delta$ in $N_\mb R$ is a set of strongly convex rational polyhedral cones in $N_\mb R$ such that (i) for all $\sigma\in\Delta$, each face of $\sigma$ is also in $\Delta$, and (ii) if $\sigma,\sigma'\in\Delta$, then $\sigma\cap\sigma'$ is a face of each. 
It gives rise to a toric variety $X(\Delta)$. 

\item  $\Delta(k):=\{\sigma\in\Delta\mid\dim\sigma=k\}$. If $\rho\in \Delta(1)$, denote the primitive generator of $\rho$ by $v_\rho\in N$ and write $G(\Delta):=\{v_\rho\mid\rho\in\Delta(1)\}$. 
\item  The dense torus $N\otimes\mf k^*$ is denoted by $T$. Let $\sigma\in\Delta$.
\begin{itemize}
\item $O_\sigma$ is the $T$-orbit corresponding to $\sigma$. Write $V_\sigma:=\overline{O_\sigma}=\bigcup_{\sigma\prec\tau}O_\tau$. In the case $\rho\in\Delta(1)$, we write $D_\rho$ instead of $V_\rho$.
\item $U_\sigma:=\on{maxSpec}\,\mf k[\chi^m\mid m\in\sigma^\vee\cap M]=\bigcup_{\tau\prec\sigma}O_\tau$.
\end{itemize} 
\item We say that $\rho,\rho'\in\Delta(1)$ are adjacent if they are contained in the same maximal cone in $\Delta$.
\item \label{essentially_surjective}A morphism $f\colon\mc F\rightarrow\mc G$ of $\mc O_X$-modules on a toric variety $X=X(\Delta)$ is called essentially surjective if $f|_T\colon\mc F|_T\rightarrow\mc G|_T$ is surjective.
\item For any $\mc O_X$-module $\mc F$, we write $\widehat{\mc F}:=\mc F^{\vee\vee}$. 

\end{enumerate}
\end{notation}

\begin{proposition}\label{prop2}
Let $X=X(\Delta)$ be a proper $\mb Q$-factorial toric variety. Then $\vep(T_X,1)\geq 0$.
\end{proposition}
\begin{proof}
By \cite[Theorem 8.1.6]{MR2810322}, we have the short exact sequence  
    \begin{equation}\label{*}\tag{2.1}
      0\rightarrow\widehat{\Omega}_X^1\rightarrow\bigoplus_{\rho\in\Delta(1)}\mathcal O_X(-D_\rho)\rightarrow \on{Cl}(X)\otimes_{\mathbb Z}\mathcal O_X\rightarrow 0.
    \end{equation}
    The dual of (\ref{*}) gives
    \begin{equation}\label{**}\tag{2.2}
      \bigoplus_{\rho\in\Delta(1)}\mathcal O_X(D_\rho)\rightarrow T_X.
    \end{equation}
     Let $\sigma\in\Delta(n)$ be a maximal cone and consider the sequence (\ref{*}) for $U_\sigma$:
     \[
       0\rightarrow\widehat{\Omega}_{U_\sigma}^1\rightarrow\bigoplus_{\rho\prec\sigma}\mathcal O_{U_\sigma}(-D_\rho|_{U_\sigma})\rightarrow \on{Cl}(U_\sigma)\otimes_{\mathbb Z}\mathcal O_{U_\sigma}\rightarrow 0.
     \]
     Using the facts that
     \begin{itemize}
         \item for any integral domain $R$, $\on{Ext}^1_{R}(R/mR,R)=0$ for any $m\in\mb Z_{\geq 0}$, and
         \item $\on{Cl}(U_\sigma)$ is torsion,
     \end{itemize}
     we have
     $
       T_{U_\sigma}\simeq \bigoplus_{\rho\prec\sigma}\mathcal O_{U_\sigma}(D_\rho|_{U_\sigma}).
     $
     Now the restriction of (\ref{**}) to $U_\sigma$ is given by the projection
     \[
       \bigoplus_{\rho\in\Delta(1)}\mathcal O_X(D_\rho|_{U_\sigma})\rightarrow \bigoplus_{\rho\in\Delta(1),\,\rho\prec\sigma}\mathcal O_X(D_\rho|_{U_\sigma})\simeq T_{U_\sigma},
     \]
     which is surjective, and thus so is (\ref{**}) as $\{U_\sigma\mid\sigma\in\Delta(n)\}$ covers $X$. The proposition now follows from Lemma~\hyperref[lemma1-1]{\ref*{lemma1}\ref*{lemma1-1}}, \hyperref[lemma1-2]{\ref*{lemma1-2}}, and the fact that $D_\rho\cdot C\geq 0$ for any $\rho\in\Delta(1)$ and any irreducible curve $C$ passing through $1$.   
\end{proof}

\begin{definition}[the condition $(\dagger)$]\label{dagger}
Let $\Delta$ be a complete simplicial fan in $N$ of rank $n$. The condition $(\dagger)$ is defined as follows:
 if $a_1v_1+\cdots +a_kv_k=0$
  for some $\{v_1,\cdots, v_k\}\subseteq G(\Delta)$ and $a_i\in\mb R_{> 0}$, then $k\geq n+1$. See Proposition~\hyperref[mmp1]{\ref*{mmp}\ref*{mmp1}} for the geometric meaning.
\end{definition}

\begin{proposition}\label{a}
$\vep(T_{X},1)=0$ if \hyperref[dagger]{$(\dagger)$} is not satisfied.
\end{proposition}
\begin{proof}
We may assume that
\[
n_1v_1+\cdots+n_{r}v_{r}=0,\ n_i\in\mb N\ \text{for all}\ i\ \text{and}\ r\leq n.
\]
Then $\{v_1,\cdots,v_r\}$ spans a subspace $V\subseteq N_\mb R$ of dimension $\ell<n$. 
Let $\Delta_0$ be the non-complete fan in $(N\cap V)_\mb R$ consisting of $0,\rho_1,\cdots,\rho_r$, where $\rho_i$ is the ray generated by $v_i$. 
We have the following maps of fans
\[
  \Delta_0\,(\text{as a fan in }(N\cap V)_\mb R) \to \Delta_0\,(\text{as a fan in }N_\mb R) \to \Delta, 
\]
where the first map is given by the natural map $N\cap V\to N$. Thus, we have
\begin{equation}\label{3}\tag{2.3}
Y(\Delta_0)\hookrightarrow Y(\Delta_0)\times (\mf k^*)^{n-\ell}\hookrightarrow X(\Delta).
\end{equation}

The following construction is taken from the proof of \cite[Proposition 2]{Pay2}.
Let $\phi_i\colon\mf k^*\rightarrow (N\cap V)\otimes \mf k^*$ be the one-parameter subgroup corresponding to $v_i$ for $1\leq i\leq r$. Consider the following map:
\[
\phi\colon\mf k^*\setminus\{\lambda_1,\ldots,\lambda_r\}\rightarrow (N\cap V)\otimes \mf k^*,\ \phi(t)=\prod_{i=1}^r\phi_i(t-\lambda_i)^{n_i},
\]
where $\lambda_1,\ldots,\lambda_r$ are distinct elements in $\mf k^*$.
Consider $\mf k^*$ as the subset $\{[a,1]\mid a\in\mf k^*\}\subseteq \mb P^1$.
Then $\phi$ can be completed into a morphism $\bar{\phi}\colon\mb P^1\rightarrow X(\Delta)$ via (\ref{3}) such that
\begin{itemize}
\item  $\bar{\phi}([1,0])=1$,
\item $\bar{\phi}([a,1])=\phi(a)\in Y(\Delta_0)$ for $a\in\mf k\setminus\{\lambda_1,\cdots,\lambda_r\}$, and
\item $\bar{\phi}([\lambda_k,1])=\left(\prod_{1\leq i\leq r,\, i\neq k}\phi_i(\lambda_k-\lambda_i)^{n_i}\right)\cdot x_{\rho_k}\in O_{\rho_k}\subseteq Y(\Delta_0)$ where $x_{\rho_k}$ is the distinguished point corresponding to $\rho_k$.
\end{itemize}
Now let $C:=\bar{\phi}(\mb P^1)$.
Then $C$ lies in the fiber of $1\in (\mf k^*)^{n-\ell}$ of $Y(\Delta_0)\times (\mf k^*)^{n-\ell}\rightarrow (\mf k^*)^{n-\ell}$, which is a trivial family of smooth varieties. We thus have 
$
\nu^* T_X\simeq \mc O_{\mb P^1}^{n-\ell}\oplus E  
$
where $\nu\colon\mb P^1\rightarrow C$ is the normalization and $E$ is some vector bundle on $\mb P^1$.
Hence $\vep(T_X,1)\leq 0$ by Lemma~\hyperref[lemma1-1]{\ref*{lemma1}\ref*{lemma1-1}} and Definition~\ref{Def:Seshadri_Constant_Proper}. Finally, by Proposition~\ref{prop2} we must have $\vep(T_X,1)=0$. 
\end{proof}

If $\pi\colon Y\rightarrow X$ is the blow-up along a smooth subvariety $Z\subseteq X\setminus\on{Sing}(X)$ and $E\subseteq Y$ is the exceptional divisor, there is a naturally defined morphism $\pi^* T_{X}(-E)\rightarrow T_{Y}$ \cite[4.10.2]{FM21}. 
The following is the toric analog, which can be used to estimate the Seshadri constants of $T_Y$ in terms of those of $T_X$ and $E$.

\begin{lemma}\label{lemma2}
Suppose 
\[
\begin{tikzcd}
Y\arrow[r,"\pi"]=Y(\widetilde{\Delta})  &	X=X(\Delta)  		
\end{tikzcd}
\] 
is a birational toric morphism between proper $\mb Q$-factorial toric varieties such that $Y$ is smooth. Let $E$ be an effective $T$-invariant divisor on $Y$ such that $\on{exc}(\pi)\subseteq\on{Supp}(E)$. Then there is a naturally defined morphism $\pi^*T_X(-mE)\rightarrow T_{Y}$ for all sufficiently large $m\in\mb Z_{>0}$. In particular, it is essentially surjective in the sense of Notation~\hyperref[essentially_surjective]{\ref*{notaion}\ref*{essentially_surjective}}.
\end{lemma}

\begin{proof}
  It suffices to prove the following case: $\sigma_2\subseteq \sigma_1\subseteq N_\mb R$ are top-dimensional, strongly convex, rational, simplicial cones, $X=U_{\sigma_1}, Y=U_{\sigma_2}$, and $\pi\colon Y\to X$ is given by the identity map $N\to N$. In this case, $\widehat{\Omega}^1_{\mf k[\sigma_2^\vee\cap M]/\mf k}$ is generated by $d\chi^{u'}$ where $u'\in\sigma_2^\vee\cap M$ (see \cite[\S8.1]{MR2810322}), and $\widehat{\Omega}^1_{\mf k[\sigma_1^\vee\cap M]/\mf k}\otimes_{\mf k[\sigma_1^\vee\cap M]}\mf k[\sigma_2^\vee\cap M]$ is generated by $d\chi^{u'}$ where $u'\in\sigma_1^\vee\cap M$. 
  We employ the following notation:
  \begin{itemize}
    \item $\sigma_2=\on{Cone}(u_1,\ldots,u_n)$ where $\{u_1,\ldots,u_n\}$ is a basis of $N$;
    \item $\sigma_2^\vee=\on{Cone}(u'_1,\ldots,u'_n)$ where $\{u'_1,\ldots,u'_n\}$ is the basis of $M$ dual to $\{u_1,\ldots,u_n\}$;
    \item $\sigma_1=\on{Cone}(v_1,\ldots,v_n)$ where each $v_i\in N$ is primitive;
    \item $\sigma_1^\vee=\on{Cone}(v'_1,\ldots,v'_n)$ where each $v'_i\in M$ is primitive and $\langle v'_i,v_j\rangle>0\Leftrightarrow i=j$ ;
    \item The rays which are common faces of $\sigma_1$ and $\sigma_2$ are $\mb R_{\geq 0}u_1,\ldots,\mb R_{\geq 0}u_k$. Hence we may assume $u_i=v_i$ for $1\leq i\leq k$ after reindexing. These rays correspond to non-$\pi$-exceptional torus invariant prime divisors.
  \end{itemize}
  Our goal is to prove that each $d\chi^{u'_\ell}$ is a linear combination of $d\chi^{v'}, v'\in\sigma_1^\vee\cap M$ with coefficients being rational functions on $Y$ whose poles are bounded by $mE$ for sufficiently large $m$. Hence $\widehat{\Omega}^1_{Y}\subseteq (\pi^*\widehat{\Omega}^1_{X})^{\vee\vee}(mE)$, and the required morphism is just the composition
  \[
     \pi^*T_X(-mE)\rightarrow (\pi^*\widehat{\Omega}^{1}_X)^\vee(-mE)=(\pi^*\widehat{\Omega}^{1}_X)^{\vee\vee\vee}(-mE)\rightarrow (\widehat{\Omega}_Y^1)^\vee=T_Y.
  \]

  Suppose $1\leq\ell\leq k$. Let $w'=v'_{k+1}+\cdots+v'_{n}$ and $\tilde v'_\ell=u'_\ell+qw'$ where $q\in \mb Z_{>0}$ is large. Then $\tilde v'_\ell\in \sigma_1^\vee$ since $\langle \tilde v'_\ell, v_j\rangle=\langle u'_\ell, v_j\rangle+\langle qw',v_j\rangle=\langle u'_\ell, u_j\rangle=\delta_{\ell j}$ for $1\leq j\leq k$, and $\langle \tilde v'_\ell, v_j\rangle=\langle u'_\ell, v_j\rangle+q\langle v'_j,v_j\rangle>0$ for $k+1\leq j\leq n$. Moreover, $v''_1=v'_1,\ldots,v''_\ell=\tilde v'_\ell,\ldots,v''_n=v'_n$ are linearly independent since the matrix $S$ defined by $S_{ij}=\langle v''_i,v_j\rangle$ is of full rank. Write $u'_\ell=\sum_{j=1}^n a_{\ell j}v''_j$ where $a_{\ell j}\in\mb Q$. We claim that $a_{\ell j}\neq 0$ only when $j=\ell$ or $k+1\leq j\leq n$. Indeed, we have $\langle u'_\ell,v_\beta\rangle=0=\langle v''_\alpha,v_\beta\rangle$ for $\alpha\in A=\{\ell,k+1,k+2,\ldots,n\}$ and $1\leq \beta\leq n,\beta\not\in A$. Since $v''_\ell,v''_{k+1},v''_{k+2},\ldots,v''_n$ are linearly independent and $\dim (\bigcap_{1\leq \beta\leq k, \beta\neq\ell} v_\beta^{\perp})=n-k+1$, we have $u'_\ell\in \bigcap_{1\leq \beta\leq k, \beta\neq\ell} v_\beta^{\perp}=\on{Span}_{\mb R}(v''_\ell,v''_{k+1},v''_{k+2},\ldots,v''_n)$. Observe that 
  $\frac{d\chi^{u'_\ell}}{\chi^{u'_\ell}}=\sum_{j=1}^na_{\ell j}\frac{d\chi^{v''_j}}{\chi^{v''_j}}$, or equivalently,
  \[
    d\chi^{u'_\ell}=\sum_{j=1}^na_{\ell j} \chi^{u'_\ell-v''_j} d\chi^{v''_j}.
  \]
  One can check by direct computation that if $a_{\ell j}\neq 0$, then $\langle u'_\ell-v''_j, u_{t}\rangle\geq 0$ for $1\leq t \leq k$.

  Suppose $k+1\leq \ell \leq n$. Then $\langle u'_\ell, v_\beta\rangle=0$ for $1\leq\beta\leq k$.
  That is, $u'_\ell\in \bigcap_{1\leq\beta\leq k} v_\beta^{\perp}=\on{Span}_\mb R(v'_{k+1},\ldots,v'_n)$, and by the same argument, we have
  \[
    d\chi^{u'_\ell}=\sum_{j=k+1}^na_{\ell j} \chi^{u'_\ell-v'_j} d\chi^{v'_j}.
  \]
  Hence $\langle u'_\ell-v'_j, u_{t}\rangle = 0$ when $a_{\ell j}\neq 0$ and $1\leq t \leq k$.
\end{proof}

\begin{lemma}\label{lemma3}
Let $f\colon X\dashrightarrow Y$ be a toric birational map between proper $\mb Q$-factorial toric varieties. Suppose that there exist an effective, $T$-invariant Cartier divisor $D$ on $X$ and a sequence of irreducible curves $\{1\in C_k\subseteq X\}_{k\in\mb N}$ such that
\begin{enumerate}[label={\upshape(\arabic*)}]
\item
$\on{exc}(f)\subseteq\on{Supp}(D)$ and 
\item
$\displaystyle\lim\limits_{k\rightarrow\infty}\frac{C_k\cdot D}{\on{mult}_1(C_k)}=\lim\limits_{k\rightarrow\infty}\frac{\overline{\mu}_{\min}(\nu_k^*T_X)}{\on{mult}_1(C_k)}=0$, where $\nu_k\colon\widetilde{C_k}\rightarrow C_k$ is the normalization.
\end{enumerate} 
Then $\vep(T_Y,1)=0$. 
\end{lemma}
Note that (2) is stronger than merely assuming $\vep(D,1)=\vep(T_X,1)=0$; in (2) we are assuming that the Seshadri constants are zero and that they can be achieved by the same sequence of irreducible curves.
\begin{proof}
Let $W$ be the smooth proper toric variety given by a common simplicial refinement of the fans of $X$ and $Y$. Then it gives the following commutative diagram of toric maps
\[
\begin{tikzcd}
&W\arrow[dr,"\beta"]\arrow[dl,"\alpha"']&\\
X\arrow[rr,dashed,"f"]&&Y
\end{tikzcd}.
\]

Let $E=\alpha^*D$. By Lemma~\hyperref[lemma1-3]{\ref*{lemma1}\ref*{lemma1-3}}, we have $\vep(T_Y,1)=\vep(\beta^*T_Y,1)$.  
Lemma~\ref{lemma2} gives the essentially surjective morphism $\beta^*T_Y\rightarrow T_W(mE)$, which implies $\vep(T_W(mE),1)\geq\vep(\beta^*T_Y,1)$ by Lemma~\hyperref[lemma1-2]{\ref*{lemma1}\ref*{lemma1-2}}. From the other essentially surjective morphism $T_W(mE)\rightarrow\alpha^*T_X(mE)$ we get $\vep(\alpha^*T_X(mE),1)\geq\vep(T_W(mE),1)$.
Decompose $\nu_k\colon\widetilde{C_k}\rightarrow C_k$ into $\widetilde{C_k}\overset{\overline{\nu}_k}{\rightarrow}\alpha^{-1}_*C_k\rightarrow C_k$
and note that, because of assumption (2),
\[
\frac{\overline{\mu}_{\min}(\overline{\nu}_k^*\alpha^*T_X(mE))}{\on{mult}_1(\alpha^{-1}_*C_k)}=\frac{\overline{\mu}_{\min}(\nu_k^*T_X(mD))}{\on{mult}_1(C_k)}\rightarrow 0\ \text{as}\ k\rightarrow\infty.
\]
Putting all these together, we conclude $\vep(T_Y,1)\leq 0$ and hence $\vep(T_Y,1)=0$ by Proposition~\ref{prop2}.
\end{proof}

\begin{proposition}\label{mmp}
Let $X=X(\Delta)$ be a proper $\mb Q$-factorial toric variety of dimension $n$ such that $\Delta$ satisfies \hyperref[dagger]{$(\dagger)$}.
\begin{enumerate}[label={\upshape(\arabic*)}]
\item  \label{mmp1} There is no surjective toric morphism from $X$ to a toric variety $Y(\Delta')$ such that $\dim X>\dim Y(\Delta')>0$.
\item  Suppose $\vep(D_\rho,1)=0$. Then there is a toric birational contraction $f\colon X\dashrightarrow X'$ which contracts $D_\rho$ and $\on{exc}(f)\subseteq D_\rho$.
\end{enumerate}
\end{proposition}

\begin{proof}

To prove (1), suppose on the contrary that we have such a morphism $f\colon X(\Delta)\rightarrow Y(\Delta')$. Let $\Delta$ and $\Delta'$ be fans in $N$ and $N'$, respectively. Then $f$ corresponds to a surjective map between lattices $\widetilde{\phi}\colon N\rightarrow N'$ such that the induced surjective map $\phi\colon N_{\mb R}\rightarrow N'_{\mb R}$ satisfies the following:
for any $\sigma\in\Delta$, we have $\phi(\sigma)\subseteq \sigma'$ for some $\sigma'\in\Delta'$. If there exists $\sigma\in\Delta$ such that $\sigma\not\subseteq \on{ker}(\phi)$ and $\on{ker}(\phi)\cap\on{reint}(\sigma)\neq\emptyset$, then $\phi(\sigma)$ is not strongly convex and cannot be contained in any cone in $\Delta'$. Hence $\{\sigma\in\Delta\mid\sigma\subseteq\on{ker}(\phi)\}$ is a complete fan in $N\cap\on{ker}(\phi)$. By assumption $\on{ker}(\phi)$ is not $0$ nor $N_\mb R$, which violates \hyperref[dagger]{$(\dagger)$}.

We now prove (2).
By running $D_\rho$-MMP, we only need to exclude the existence of a toric birational contraction $g\colon X(\Delta)\dashrightarrow X''(\Delta'')$ such that $\on{exc}(g)\subset D_\rho$, $D_\rho$ is not contracted by $g$, and either (i) $X''$ is a Mori fiber space, or (ii) $\widetilde{D}_\rho=g_*D_\rho$ is nef. Note that $\Delta''$ also satisfies \hyperref[dagger]{$(\dagger)$}.
In the case of (i),  by (1), $X''$ must be a Mori fiber space over a point. It is impossible since $\widetilde{D}_\rho$ cannot be anti-ample.
In the case of (ii), let $h\colon X''\rightarrow Z$ be the toric morphism defined by $|\widetilde{D}_\rho|$. By (1) and the fact $\widetilde{D}_\rho\not\sim 0$, we have $\dim Z=n$. Hence $h$ is birational and $\widetilde{D}_\rho=g^*A$ for some ample divisor $A$ on $Z$. By Lemma~\hyperref[lemma1-3]{\ref*{lemma1}\ref*{lemma1-3}}, $\vep(X'',\widetilde{D}_\rho;1)=\vep(Z,A;1)>0$. 
Taking a common simplicial refinement of $\Delta$ and $\Delta''$, we  get the following commutative diagram
\[
\begin{tikzcd}
&W\arrow[dr,"\beta"]\arrow[dl,"\alpha"']&\\
X\arrow[rr,dashed,"f"]&&X''
\end{tikzcd}.
\]
Now there exists $m\in\mb N$ such that $\alpha^*(mD_\rho)-\beta^*\widetilde{D}_\rho$ is an effective torus invariant divisor. We conclude that, again by Lemma~\hyperref[lemma1-3]{\ref*{lemma1}\ref*{lemma1-3}},
\[
0=\vep(mD_\rho, 1)=\vep(\alpha^*(mD_\rho), 1)\geq \vep(\beta^*\widetilde{D}_\rho, 1)=\vep(\widetilde{D}_\rho, 1),
\] 
which is a contradiction.
\end{proof}

\begin{proposition}\label{b}
$\vep(T_{X},1)>0$ if \hyperref[dagger]{$(\dagger)$} is satisfied.
\end{proposition}
\begin{proof}
Suppose on the contrary that the fan $\Delta$ of $X$ satisfies \hyperref[dagger]{$(\dagger)$} and $\vep(T_{X},1)=0$. 
Recall that we have the surjective morphism
\[
\bigoplus_{\rho\in\Delta(1)}\mc O_X(D_\rho)\rightarrow T_X\rightarrow 0.
\]
Let $C$ be an irreducible curve on $X$ passing through $1$. Then we have the exact sequence
\[
\nu^*(\bigoplus_{\rho\in\Delta(1)}\mc O_X(D_\rho))\rightarrow \nu^*T_X\rightarrow 0
\]
where $\nu\colon \widetilde{C}\rightarrow C$ is the normalization. It follows that 
\[
\frac{\overline{\mu}_{\min}(\nu^*T_X)}{\on{mult}_1(C)}\geq \frac{\overline{\mu}_{\min}(\nu^*(\bigoplus_{\rho\in\Delta(1)}\mc O_X(D_\rho)))}{\on{mult}_1(C)}=\frac{\min_{\rho\in\Delta(1)}\{C\cdot D_\rho\}}{\on{mult}_1(C)}. 
\]
Hence, we can choose a sequence of irreducible curves $\{1\in C_k\subseteq X\}_{k\in\mb N}$ such that, for some $\rho\in\Delta(1)$,
\[
\lim\limits_{k\rightarrow\infty}\frac{C_k\cdot D_\rho}{\on{mult}_1(C_k)}=\lim\limits_{k\rightarrow\infty}\frac{\overline{\mu}_{\min}(\nu_k^*T_X)}{\on{mult}_1(C_k)}=0,
\]
where $\nu_k\colon\widetilde{C_k}\rightarrow C_k$ is the normalization. Thus $\vep(X,D_\rho)=0$ and by Proposition~\ref{mmp}, we have a toric birational contraction $f\colon X\dashrightarrow X'$ such that $\on{exc}(f)\subseteq\on{Supp}(D_\rho)$ and $\rho(X)>\rho(X')$. Then the fan of $X'$ also satisfies \hyperref[dagger]{$(\dagger)$} and $\vep(T_{X'},1)=0$ by Lemma~\ref{lemma3}. We get a contradiction since this process cannot be carried out for infinitely many times.
\end{proof}

\begin{theorem}\label{combinatoric}
Suppose that $X=X(\Delta)$ is a proper $\mb Q$-factorial toric variety. Then $\vep(T_X,1)>0$ if and only if $\Delta$ satisfies \hyperref[dagger]{$(\dagger)$} in Definition~\ref{dagger}.
\end{theorem}
\begin{proof}
It's just the combination of Proposition~\ref{a} and Proposition~\ref{b}.
\end{proof}

\section{Proof of Theorem~\ref{thm1}} 
\begin{proposition}\label{splitting}
If $X=X(\Delta)$ is a smooth projective toric variety of dimension $n$ and $Y\subseteq X$ is a $T$-invariant prime divisor, then the short exact sequence 
\[
0\rightarrow T_Y\rightarrow T_X|_Y\rightarrow N_{Y/X}\rightarrow 0
\]
splits.
\end{proposition}
\begin{proof}
In this proof, we use the terms "invertible sheaves" and "line bundles" interchangeably. We have $Y=D_\rho$ for some $\rho\in\Delta(1)$.
It is enough to show that $T_X|_Y$ contains a line sub-bundle $L$ such that $L_y\nsubseteq T_{Y,y}$ for any $y\in Y$. For any $\sigma\in\Delta(n)$ containing $\rho$,
from the proof of Proposition~\ref{prop2}, we have a natural map $\mc O_{U_\sigma}(D_{\rho}|_{U_\sigma})\to T_{U_\sigma}$ which fits into the decomposition 
\[
  T_{U_\sigma}=\bigoplus_{\rho'\prec\sigma,\,\rho'\in\Delta(1)}\mc O_{U_\sigma}(D_{\rho'}|_{U_\sigma}).
\]
Let $U:=\bigcup_{\rho\prec\sigma\in\Delta(n)} U_\sigma$. Then $Y\subseteq U$ and the above maps glue together to give a line sub-bundle $\mc O_U(D_\rho|_U)\to T_U$. If we set $L:=\mc O_U(D_\rho|_U)\otimes_{\mc O_U}\mc O_Y$, then $L$ is a line sub-bundle of $T_X|_Y$.  For any $\sigma\in\Delta(n)$ containing $\rho$, we have $U_\sigma\simeq\mb A^n$ and $D_\rho\vert_{U_\sigma}=\{x_1=0\}$ after reindexing; the image of $\mc O_{U_\sigma}(D_\rho\vert_{U_\sigma})\to T_{U_\sigma}$ is given by $\frac{\partial}{\partial x_1}$. Then it is clear that $L_y\nsubseteq T_{Y,y}$ for any $y\in Y$.
\end{proof}

\begin{definition}
A non-empty subset $\mc B=\{v_1,\cdots,v_k\}\subseteq G(\Delta)$ is called a primitive collection if for each $i$, $\mc B\setminus\{v_i\}$ generates a ($k-1$)-dimensional cone in $\Delta$, while $\mc B$ does not generate a $k$-dimensional cone in $\Delta$. 
\end{definition}

\begin{proposition}[{\cite[Proposition 3.2]{MR1133869}}]\label{primitive}
Suppose that $X=X(\Delta)$ is a smooth projective toric variety. Then there exists a primitive collection $\mc B=\{v_1,\cdots,v_k\}$ such that $v_1+\cdots +v_k=0$.
\end{proposition} 

\begin{proof}[Proof of Theorem~\ref{thm1}]
Suppose $p\in T$. We may assume $p=1$.
Let $\mc B=\{v_1,\cdots,v_k\}\subseteq G(\Delta)$ be a primitive collection such that $v_1+\cdots +v_k=0$, whose existence is guaranteed by Proposition~\ref{primitive}. If $k\leq n$, then $\Delta$ does not satisfy \hyperref[dagger]{$(\dagger)$} and by Theorem \ref{combinatoric} we must have $\vep(T_X,1)=0$ . Hence $k=n+1$ and $X\simeq\mb P^n$.

Suppose $p\in X\setminus T$ and thus $p\in D_\rho$ for some $\rho\in\Delta(1)$.
We have the splitting short exact sequence from Proposition~\ref{splitting}
\[
0\rightarrow T_{D_\rho}\rightarrow T_X|_{D_\rho}\rightarrow N_{D_{\rho}/X}\rightarrow 0.
\]
Lemma~\hyperref[lemma1-2]{\ref*{lemma1}\ref*{lemma1-2}} then gives
\[
\vep(T_X,p)\leq\vep(T_X|_{D_\rho},p)=\min\{\vep(T_{D_\rho},p),\vep(N_{D_{\rho}/X},p)\}.
\]
By induction on $\dim X$, we may assume that $D_\rho\simeq\mb P^{n-1}$ and $N_{D_{\rho}/X}$ is ample. 

Let $\sigma=\langle v_1,\cdots,v_n=v_\rho\rangle\in\Delta(n)$. Then all the rays adjacent to $\rho$ other than $\rho$ itself are $\langle v_1\rangle,\cdots,\langle v_{n-1}\rangle$   
and $\langle v_{n+1}\rangle$, where $v_{n+1}=-v_1-\cdots-v_{n-1}-\ell v_n$ for some $\ell\in\mb Z$. Since $N_{D_\rho/X}\simeq\mc O_{D_\rho}(D_\rho|_{D_\rho})$ is ample, we must have $\ell>0$ and thus all the other rays in $\Delta$ are contained in $\langle v_1,\cdots,v_{n-1},v_{n+1}\rangle$ (see Figure 1).
Let $\mc B=\{w_1,\cdots,w_k\}\subseteq G(\Delta)$ be a primitive collection as in Proposition~\ref{primitive}. The condition $w_1+\cdots+w_k=0$  implies that $v_n\in\mc B$.\\ 
\textbf{Case 1}: $k\geq 3$.
From the definition of primitive collections, $\langle w_i\rangle$ is adjacent to $\rho$ for each $i$ and hence $\mc B\subseteq \{v_1,\cdots,v_{n+1}\}$. Since $\overline{w}_1+\cdots+\overline{w}_k=0$ in $N_{\mb R}/\langle\pm\rho\rangle$, we conclude that $\mc B=\{v_1,\cdots,v_{n+1}\}$ and hence $X\simeq\mb P^n$. \\
\textbf{Case 2}: $k=2$.
We have $-v_n\in \mc B\subseteq G(\Delta)$ and hence $-\rho\in \Delta$. By using the argument similar to that in Proposition~\ref{a} we conclude $p\notin O_\rho$. Consequently,  $p\in D_{\rho'}$ for some $\rho'=\langle v_i\rangle,\ i\neq n$. But the same argument implies $v_i\in\mc B$, which is absurd.
\end{proof}

\begin{figure}
\centering
\tikzset{every picture/.style={line width=0.75pt}} 
\begin{tikzpicture}[x=0.75pt,y=0.75pt,yscale=-1,xscale=1]
\draw [fill=lightgray][color={rgb, 255:red, 200; green, 200; blue, 200 }  ][line width=0.75] (69,127)--(78.88,148)--(148,103.5)--(69,127);
uncomment if require: \path (0,235); 

Straight Lines [id:da08157189792382913] 
\draw    (100,103) -- (196,103.12) ;
\draw [shift={(198,103.12)}, rotate = 180.07] [color={rgb, 255:red, 0; green, 0; blue, 0 }  ][line width=0.75]    (10.93,-3.29) .. controls (6.95,-1.4) and (3.31,-0.3) .. (0,0) .. controls (3.31,0.3) and (6.95,1.4) .. (10.93,3.29)   ;
Straight Lines [id:da2606257355924233] 
\draw    (100,103) -- (100,22.12) ;
\draw [shift={(100,20.12)}, rotate = 90] [color={rgb, 255:red, 0; green, 0; blue, 0 }  ][line width=0.75]    (10.93,-3.29) .. controls (6.95,-1.4) and (3.31,-0.3) .. (0,0) .. controls (3.31,0.3) and (6.95,1.4) .. (10.93,3.29)   ;
Straight Lines [id:da4181995678327064] 
\draw    (100,103) -- (41.6,146.92) ;
\draw [shift={(40,148.12)}, rotate = 323.06] [color={rgb, 255:red, 0; green, 0; blue, 0 }  ][line width=0.75]    (10.93,-3.29) .. controls (6.95,-1.4) and (3.31,-0.3) .. (0,0) .. controls (3.31,0.3) and (6.95,1.4) .. (10.93,3.29)   ;
Straight Lines [id:da33594745074415977] 
\draw    (100,103) -- (91.81,119.72) ;
\draw  [dashed]  (78.88,148) --(91.81,119.72);
\draw    (78.88,148) -- (50.88,203.32) ;

\draw [shift={(50,205.12)}, rotate = 296.09] [color={rgb, 255:red, 0; green, 0; blue, 0 }  ][line width=0.75]    (10.93,-3.29) .. controls (6.95,-1.4) and (3.31,-0.3) .. (0,0) .. controls (3.31,0.3) and (6.95,1.4) .. (10.93,3.29)   ;
Shape: Circle [id:dp7362082900248492] 
\draw  [fill={rgb, 255:red, 0; green, 0; blue, 0 }  ,fill opacity=1 ] (98.56,63) .. controls (98.56,62.2) and (99.2,61.56) .. (100,61.56) .. controls (100.8,61.56) and (101.44,62.2) .. (101.44,63) .. controls (101.44,63.8) and (100.8,64.44) .. (100,64.44) .. controls (99.2,64.44) and (98.56,63.8) .. (98.56,63) -- cycle ;
Shape: Circle [id:dp3981691698264984] 
\draw  [fill={rgb, 255:red, 0; green, 0; blue, 0 }  ,fill opacity=1 ] (146.12,103.06) .. controls (146.12,102.26) and (146.76,101.62) .. (147.56,101.62) .. controls (148.36,101.62) and (149,102.26) .. (149,103.06) .. controls (149,103.86) and (148.36,104.5) .. (147.56,104.5) .. controls (146.76,104.5) and (146.12,103.86) .. (146.12,103.06) -- cycle ;
Shape: Circle [id:dp1921016978384935] 
\draw  [fill={rgb, 255:red, 0; green, 0; blue, 0 }  ,fill opacity=1 ] (67.12,126.56) .. controls (67.12,125.76) and (67.76,125.12) .. (68.56,125.12) .. controls (69.36,125.12) and (70,125.76) .. (70,126.56) .. controls (70,127.36) and (69.36,128) .. (68.56,128) .. controls (67.76,128) and (67.12,127.36) .. (67.12,126.56) -- cycle ;
Shape: Circle [id:dp1439865891217127] 
\draw  [fill={rgb, 255:red, 0; green, 0; blue, 0 }  ,fill opacity=1 ] (77,147.56) .. controls (77,146.76) and (77.64,146.12) .. (78.44,146.12) .. controls (79.24,146.12) and (79.88,146.76) .. (79.88,147.56) .. controls (79.88,148.36) and (79.24,149) .. (78.44,149) .. controls (77.64,149) and (77,148.36) .. (77,147.56) -- cycle ;
\draw [color={rgb, 255:red, 136; green, 136; blue, 136 }](100,103)--(106,124)--(110,138);
\filldraw [color={rgb, 255:red, 136; green, 136; blue, 136 }](106,124) circle (1pt);
\draw [color={rgb, 255:red, 136; green, 136; blue, 136 }](116,159)--(122,180);
\draw [color={rgb, 255:red, 136; green, 136; blue, 136 }](118,176)--(122,180)--(123.5,174.5);
\draw (122,188) node [anchor=north west][inner sep=0.75pt]  [font=\footnotesize,rotate=-359] [align=left] {$\displaystyle \rho'$: not adjacent to $\rho$};

\draw (107.79,2.14) node [anchor=north west][inner sep=0.75pt]  [font=\footnotesize,rotate=-359] [align=left] {$\displaystyle \rho $};
\draw (50,112) node [anchor=north west][inner sep=0.75pt]  [font=\footnotesize]  {$v_{1}$};
\draw (103,52) node [anchor=north west][inner sep=0.75pt]  [font=\footnotesize]  {$v_{n}$};
\draw (151,87) node [anchor=north west][inner sep=0.75pt]  [font=\footnotesize]  {$v_{i}$};
\draw (90,140) node [anchor=north west][inner sep=0.75pt]  [font=\footnotesize]  {$v_{n+1} =-v_{1} -\cdots -v_{n-1} -\ell v_{n}$};
\end{tikzpicture}
\caption{The rays adjacent to $\rho$. The images of $\langle v_1\rangle,\ldots,\langle v_{n-1}\rangle,\langle v_{n+1}\rangle$ in $N_{\mb R}/\langle v_n\rangle$ form the $1$-skeleton of the fan of $\mb P^{n-1}$. }
\label{pic1}
\end{figure}
The following example demonstrates that the conjecture fails  when the variety is not smooth, even in the toric case:
 \begin{example}\label{Example:terminal_threefold}
   Let $N=\mb Z^3$ and consider the fan $
   \Delta$ generated by $(1,0,0)$, $(0,1,0)$, $(0,0,1)$ and $(-1,-2,-3)$.
   Then $X(\Delta)$ is isomorphic to the weighted projective space $\mb P(1,1,2,3)$, which has terminal singularities.
We have $\vep(T_X,1)>0$ by Theorem~\ref{combinatoric} while $X\not\simeq\mb P^3$.
 \end{example}
\section{Some formulas for Seshadri constants}
\begin{proposition}\label{prop3}
   Suppose $X=X(\Delta)$ is a smooth projective toric variety. Then for any $p\in X$, 
   \[
   \vep(T_X,p)=\min_{\rho\in\Delta(1)}\{\vep(D_\rho,p)\}.
   \]
 \end{proposition}
 \begin{proof}
 If $X\simeq\mb P^n$, then $\vep(T_{\mb P^n},p)=1=\vep(D,p)$ for any $p\in\mb P^n$ and any $T$-invariant divisor $D$ on $\mb P^n$.
Suppose $X$ is not isomorphic to the projective space. 
For any $p\in X$, the surjective morphism
$\bigoplus_{\rho\in\Delta(1)}\mathcal O_X(D_\rho)\rightarrow T_X$ from the proof of Proposition~\ref{prop2}
implies $\vep(T_X,p)\geq\min_{\rho\in\Delta(1)}\{\vep(D_\rho,p)\}$. 
If $p\in D_\rho$, the proof of Theorem~\ref{thm1} gives 
$
\vep(T_X,p)\leq\min\{\vep(T_{D_\rho},p),\vep(N_{D_\rho/X},p)\}.
$
Using $N_{D_\rho/X}\simeq \mc O_X(D_\rho)\vert_{D_\rho}$, we obtain
\[
\min_{p\in D_\rho}\{\vep(\mc O_X(D_\rho)\vert_{D_\rho},p)\}\geq\vep(T_X,p)\geq\min_{\rho\in\Delta(1)}\{\vep(D_\rho,p)\}.
\]
By Theorem~\ref{thm1} we must have $\vep(T_X,p)\leq 0$ for all $p\in X$.

If $\vep(T_X,p)<0$, then $\min_{\rho\in\Delta(1)}\{\vep(D_\rho,p)\}<0$ and thus $\min_{\rho\in\Delta(1)}\{\vep(D_\rho,p)\}=\min\{\vep(D_\rho,p)\mid \rho\in\Delta(1),\vep(D_\rho,p)<0\}$. On the other hand, if $\vep(D_\rho,p)<0$, then $p\in D_\rho$ and $\vep(\mc O_X(D_\rho)\vert_{D_\rho},p)=\vep(D_\rho,p)$. So $\vep(T_X,p)=\min\{\vep(D_\rho,p)\mid \rho\in\Delta(1),\vep(D_\rho,p)<0\}=\min_{\rho\in\Delta(1)}\{\vep(D_\rho,p)\}$.

If $\vep(T_X,p)= 0$, then $\min_{p\in D_\rho}\{\mc O_X(D_\rho)\vert_{D_\rho}\}\geq 0$ and thus $\vep(D_\rho,p)\geq 0$ whenever $p\in D_\rho$ as in the previous paragraph. Hence
$\min_{\rho\in\Delta(1)}\{\vep(D_\rho,p)\}\geq 0$, and as a result, $\vep(T_X,p)= 0=\min_{\rho\in\Delta(1)}\{\vep(D_\rho,p)\}$. 
\end{proof}

\begin{corollary}\label{cor5}
Assume as above. Then $\vep(T_X,p)$ is lower semicontinuous in $p$.
\end{corollary}
\begin{proof}
If $X\simeq \mb P^n$, then $\vep(T_X,p)=1$ is a constant. Suppose $X\not\simeq\mb P^n$.
Proposition~\ref{prop3} and Theorem~\ref{thm1} imply $
\vep(T_X,p)=\min\{0,\min_{p\in D_\rho}\{\vep(D_\rho,p)\}\}$
and the statement follows.
\end{proof} 
Note that the tangent bundle of a smooth projective toric variety is big but not necessarily nef. It would be interesting to see whether Corollary~\ref{cor5} is true for every smooth projective variety $X$ with big $T_X$.
 \begin{example}
 In the proof of Proposition~\ref{prop3}, we see that $\vep(T_X,p)=\min_{p\in D_\rho}\{\vep(D_\rho,p)\}$ if $\vep(T_X,p)<0$. In general,
 to apply Proposition~\ref{prop3}, we need to calculate $\vep(D_\rho,p)$ for every $\rho\in\Delta(1)$ no matter whether $D_\rho$ contains $p$ or not.
  Let $N=\mb Z^2$, and let $\Delta$ be the fan generated by $v_1=(1,0),v_2=(0,1),v_3=(0,-1)$, and $v_4=(-1,r)$.  Then $X(\Delta)\simeq\Sigma_r$, the $r$-th Hirzebruch surface.
 If $p$ is a general point on $D_2$, then $\vep(D_2,p)=-r$ and hence $\vep(T_X,p)<0$. As a result, $\vep(T_X,p)=\min_{p\in D_\rho}\{\vep(D_\rho,p)\}=-r$, since $D_2$ is the only $T$-invariant divisor containing $p$.
 
 Now let $p$ be a general point on $D_3$. We have $\vep(D_1,p)=\vep(D_2,p)=\vep(D_4,p)=0$. Note that $\vep(D_3,p)$ can be calculated on $Y(\Delta')$, where $\Delta'$ is the fan generated by $v_1,v_3$ and $v_4$. 
Then $\vep(D_3,p)>0$
 since $D_3'$ on $Y(\Delta')$ is ample. Hence 
$\vep(T_X,p)=0\neq \vep(D_3,p)=\min_{p\in D_\rho}\{\vep(D_\rho,p)\}$.
 \end{example}
 
 Although the above example shows that we can actually have $0>\vep(T_X,p)>-\infty$, it is quite rare in the sense of the following proposition.
 
 \begin{proposition} \label{prop4}
   Suppose $X$ is a smooth projective toric variety. If $\on{dim}(\overline{\on{orb}(p)})\geq 2$, then $\vep(T_X,p)=0$ or $-\infty$.
   In other words, the Seshadri constant is "nontrivial" (meaning that $0>\vep(T_X,p)>-\infty$) only if $\on{dim}(\overline{\on{orb}(p)})\leq 1$. 
 \end{proposition}
Note that $\on{dim}(\overline{\on{orb}(p)})\leq 1$ does not guarantee $\vep(T_X,p)>-\infty$. For example, let $X=\on{Bl}_{p}\mb P^3\overset{\pi}{\rightarrow} \mb P^3$ and let $E=\on{exc}(\pi)$. 
We have $\vep(E,q)=-\infty$ for any $q\in E$ since $E\simeq\mb P^2$ and $E|_E$ is anti-ample.
Then $\vep(T_X,q)=-\infty$ for any $q\in E$ by Proposition~\ref{prop3}.
\begin{lemma}\label{lemma7}
Let $X=X(\Delta)$ be a proper $\mb Q$-factorial toric variety such that $\dim X\geq 2$, and let $D=\sum_{\rho\in\Delta(1)} a_\rho D_\rho$ be a $T$-invariant $\mb Q$-divisor. If $\vep(D,1)<0$, then $\vep(D,1)=-\infty$.
\end{lemma}
\begin{proof}
By assumption, there exists an irreducible curve $C$ passing through $1\in T$ such that $C\cdot D<0$. Let $m_\rho=C\cdot D_\rho$. Then we have $N:=\sum_{\rho\in\Delta(1)}m_\rho a_\rho<0$. For any $k>0$, by \cite[Section 3]{Pay2} we have an irreducible curve $C_k$ such that $C_k\cap T\neq\emptyset$ and $C_k\cdot D_\rho=km_\rho$ for each $\rho\in\Delta(1)$. After applying suitable automorphism $t\in T$, we may assume $1\in C_k$ and $\on{mult}_1(C_k)=1$. Thus $\frac{C_k\cdot D}{\on{mult}_1(C_k)}=kN\rightarrow -\infty$ as $k\rightarrow\infty$.
\end{proof}
\begin{proof}[Proof of Proposition~\ref{prop4}]
Suppose $\vep(T_X,p)<0$. Then by Proposition~\ref{prop3}, there exists an irreducible curve $C$ containing $p$ such that $C\cdot D_\rho<0$ for some $\rho\in\Delta(1)$. Let $\tau\in\Delta$ be the unique cone such that $O_\tau\cap C$ is dense in $C$. 
Note that $V_\tau$ contains $C$ and $p$, and we may assume $1_\tau\in C$
after applying suitable automorphism $t\in T$.
We can write $D_\rho|_{V_\tau}\sim E$ for some $O_\tau$-invariant divisor $E$ on $V_\tau$ and thus $C\cdot E<0$. 
If $\on{dim}(\overline{\on{orb}(p)})\geq 2$, then $\dim V_\tau\geq 2$ and $\vep(E,1_\tau\in O_\tau)=-\infty$ by Lemma~\ref{lemma7}. 
Therefore $\vep(T_X,1_\tau)=-\infty$ again by Proposition~\ref{prop3}, and we have $\vep(T_X,p)=-\infty$ by semi-continuity Corollary~\ref{cor5}.  
\end{proof}

Finally, we raise a question on the necessity of the projectiveness assumption.
\begin{question}\label{projectiveness}
Is there any smooth complete fan $\Delta$, other than the fan of the projective space, satisfying the condition \hyperref[dagger]{$(\dagger)$} in Definition~\ref{dagger}?
\end{question}

In view of Proposition~\ref{primitive}, $X(\Delta)$ must be non-projective.
If such an example does exist, then it will serve as a counterexample of Conjecture~\ref{conj} without assuming that $X$ is projective. If the answer to Question~\ref{projectiveness} is no, then the proof of Theorem~\ref{thm1} can be carried over and we can generalize the theorem to smooth proper toric varieties.

\subsection*{Acknowledgements} 
The author would like to thank Sam Payne for the useful advice concerning the toric MMP. He also thanks
Jungkai Chen for the warm encouragement during the preparation of the paper. Part of the work was done during the author's stay in NCTS (National Center for Theoretical Sciences). This work is supported by NSTC (National Science and Technology Council) under Grant No. 111-2115-M-002-004-MY3 and by National Taiwan University.

\bibliographystyle{alpha}
\bibliography{mybib}
\end{document}